\documentclass{amsart}
\usepackage{amssymb}
\usepackage{mathrsfs}
\usepackage{graphicx}
\usepackage{comment}
\usepackage[francais,english]{babel}

\newcommand{\be}{\begin{equation}}
\newcommand{\ee}{\end{equation}}
\newcommand{\ba}{\begin{align}}
\newcommand{\ea}{\end{align}}

\DeclareMathOperator{\Res}{Res}

\newtheorem{theorem}{Theorem}[section]

\newtheorem{lemma}{Lemma}[section]

{\begin{list}{}{%
\settowidth{\labelwidth}{\textsf{{\it #1.}}}%
\setlength{\labelsep}{4mm}%
\setlength{\leftmargin}{\labelwidth}%
\addtolength{\leftmargin}{\labelsep}%
}}%
{\end{list}}

{\begin{list}{}{%
\settowidth{\labelwidth}{\textsf{{\it #1.}}}%
\setlength{\labelsep}{2mm}%
\setlength{\leftmargin}{\labelwidth}%
\addtolength{\leftmargin}{\labelsep}%
\addtolength{\leftmargin}{4mm}%
\setlength{\itemsep}{6pt}%
\setlength{\listparindent}{0pt}%
\setlength{\topsep}{3pt}%
}}%
{\end{list}}

\title{Minimal  group determinants and the  Lind-Lehmer problem for dihedral groups}

\author[T. Boerkoel]{Ton Boerkoel}
\address{Department of Mathematics\\
DigiPen Institute of Technology\\
Redmond, WA 98052, USA}
\email{aboerkoel@digipen.edu}

\author[C. Pinner]{Christopher Pinner}
\address{ Department of Mathematics\\
         Kansas State University\\
         Manhattan, KS 66506, USA}
\email{pinner@math.ksu.edu}
\thanks{The second  author thanks the University of Edinburgh for the invitation  to visit,  the Edinburgh Math Society for its financial  support, and Chris Smyth for some  very useful discussions.}

\date{\today}

\begin{document}


\begin{abstract}
We find the minimal non-trivial integer variable group determinant for any  dihedral group  of order less than $3.79\times 10^{47}$.
We think of this as the Lind-Lehmer problem for the dihedral group. We give a complete description of the determinants for some dihedral groups including  $D_{2p}$ and $D_{4p}$.

\end{abstract}

\maketitle

\section{Introduction}\label{secIntroduction}

For a finite group $G=\{g_1,\ldots ,g_n\}$ we assign a variable $x_g$ for each $g$ in $G$ and define the {\it group determinant} 
to be the $n\times n$ determinant
$$ \mathscr{D}_G(x_{g_1},\ldots ,x_{g_n})= \det \left( x_{g_i g_j^{-1}}\right). $$
Plainly this is a homogeneous polynomial of degree $n$ in the $n$ variables $x_g$. 
We are interested here in the smallest non-trivial value this can take when the $x_g$ are  all  integers:
\be \label{deflambda} \lambda(G) =\min\{ |\mathscr{D}_G( x_{g_1},\ldots ,x_{g_n})|\geq 2\; : \; (x_{g_1},\ldots ,x_{g_n})\in \mathbb Z^n\}. \ee
In the  cyclic group  case $\mathscr{D}_G( x_{g_1},\ldots ,x_{g_n})$ reduces to a circulant determinant. An old  problem of Olga Taussky-Todd  is to determine
what values these can take for integer variables, see for example  \cite{Newman1,Newman2,Laquer,Norbert2}. For simplicity we shall use $\mathbb Z_n$,  rather than  $\mathbb Z/n\mathbb Z,$ to denote the 
integers mod $n$, the basic cyclic group of order $n$. Studying the related Lind-Lehmer problem for cyclic groups, the value of $\lambda(\mathbb Z_n)$  was obtained by Kaiblinger  \cite{Norbert} for $n$  not a multiple of 420, with \cite{Pigno1} extending this to $n$  not a multiple of $2^3\cdot 3\cdot 5\cdot 7\cdot 11\cdot 13\cdot 17\cdot 19\cdot 23$. Here we  do the same for $D_{2n}$, the  dihedral group of order $2n$, when $n$ is not a multiple of $\displaystyle 2^2 \cdot 3^2 \cdot \prod_{5\leq p\leq 113}p > 1.89 \times 10^{47}$.

Dedekind  showed  that for  abelian  $G$ the group determinant can be factored over $\mathbb C[x_{g_1},\ldots ,x_{g_n}]$ into linear factors using the group of characters $\hat{G}$ on $G$
\be \label{abelianfactor} \mathscr{D}_G(  x_{g_1},\ldots ,x_{g_n})=\prod _{\chi \in \hat{G}} \left( \chi(g_1) x_{g_1}+\cdots + \chi(g_n)x_{g_n}\right), \ee
see for example Lang \cite[\S 3.6]{Lang}.
Dedekind  observed that for non-abelian groups one could have non-linear factors, 
leading  Frobenius to  develop representation theory to describe the factorisation. In particular if $\hat{G}$ denotes  a complete 
set of non-isomorphic  irreducible representations of $G$  we have, see for example \cite{Formanek},
\be \label{Frobenius} \mathscr{D}_G(  x_{g_1},\ldots ,x_{g_n})=\prod _{\rho \in \hat{G}} \det \left( \sum_{g\in G} x_g \rho(g) \right)^{\deg \rho}. \ee
See Conrad \cite{Conrad} for an account of  the historical development  of \eqref{abelianfactor} and \eqref{Frobenius}.
Notice that the group determinant preserves multiplication in $\mathbb Z [G]$, in the sense that if 
$$ \left(\sum_{g\in G}  a_g g \right)  \left(\sum_{g\in G}  b_g g \right) = \left(\sum_{g\in G}  c_g g \right), \;\;\; c_g =\sum_{uv=g}a_ub_v, $$
then
\be \label{mult}  \mathscr{D}_G(a_{g_1},\ldots ,a_{g_n})  \mathscr{D}_G(b_{g_1},\ldots ,b_{g_n})= \mathscr{D}_G(c_{g_1},\ldots ,c_{g_n}), \ee
since plainly $c_{g_ig_j^{-1}}=\sum_{u\in G} a_{g_i u^{-1}} b_{ug_j^{-1}}$. Finally, we note the trivial bound
$$    \lambda(G) \leq \max\{ 2,  |G|-1\} $$
since, taking $g_1$ to be the identity element, 
$$\mathscr{D}_G(0,1,\ldots ,1)=(-1)^{n-1}(n-1),$$
an easy  evaluation viewed as a Lind Mahler measure $M_{\mathbb Z_n}(-1+(x^n-1)/(x-1))$.

\section{Lind Mahler Measure}\label{MahlerMeasure}
For a polynomial $f(x)=\sum_{j=0}^N a_j x^j$ in $\mathbb Z[x]$  one defines the  traditional {\it logarithmic  Mahler measure} $m(f)$ by 
$$ m(f)=\int_0^1 \log |f(e^{2\pi i\theta})| d \theta, $$
with the classical Lehmer problem \cite{Lehmer} asking whether there is a constant  $c>0$ such that $m(f)=0$ or $m(f)\geq c$.
Lind \cite{Lind} viewed $f(e^{2\pi i\theta})= \sum_{j=0}^N a_j e^{2\pi i j\theta}$ as a linear sum of characters on the group $[0,1)=\mathbb R/\mathbb Z$, and extended the concept of Mahler measure to an arbitrary compact abelian group $G$ with Haar measure $\mu$ and group of characters $\hat{G}$, defining 
$$ m_G(f) = \int_G \log | f| d \mu $$
where $f$ is an element in $\mathbb Z [\hat{G}]$. For example, for a finite abelian group
\be \label{defG} G=\mathbb Z_{n_1} \times \cdots \times \mathbb Z_{n_r} \ee
the characters take the form
$$ \hat{G}= \{\chi_{u}(x_1,\ldots ,x_r)= e^{2\pi i x_1 u_1/n_1}\cdots e^{2\pi i x_r u_r/n_r}\; : \; u=(u_1,\ldots ,u_r)\in G\}, $$
the integral becoming an average over the group $G$
\begin{align*} m_G\left(\sum_{u\in G} a_u \chi_u\right) & = \frac{1}{|G|} \sum_{x=(x_1,\ldots ,x_r)\in G} \log\left|\sum_{u\in G} a_u \chi_u(x_1,\ldots ,x_r)\right| \\
 & = \frac{1}{|G|} \log \prod_{x\in G} \left|\sum_{u\in G} a_u e^{2\pi i u_1x_1/n_1} \cdots e^{2\pi i u_rx_r/n_r}\right|\\
 & =  \frac{1}{|G|} \log \prod_{x\in G} \left|\sum_{u\in G} a_u \chi_x(u) \right|\\
 & =  \frac{1}{|G|}|\mathscr{D}_G(a_{g_1},\ldots ,a_{g_n})|.
\end{align*}
Thus, as observed by Vipismakul \cite{Cid1} in his thesis, the Lind-Mahler measure for a finite abelian group essentially corresponds to evaluating the group determinant. In particular  the corresponding Lehmer problem for the group, that is
determining the  minimal Lind-Mahler measure $m(f)>0$, is the same as finding $\frac{1}{|G|} \log \lambda (G)$. 

As mentioned above cyclic groups were considered in \cite{Lind, Norbert, Pigno1},  with  $G=\mathbb Z_p^r$ and various other groups  investigated  in \cite{dilum, Cid2, pgroups, Stian}. In those papers  the Lind-Mahler measure was mostly viewed as the measure of a polynomial, representing the average of  the polynomial over appropriate roots of unity; that is for a group  $G$ of the form \eqref{defG} and polynomial  $F$ in $\mathbb Z[x_1,\ldots ,x_r]$, we define
$$ m_G(F) =\frac{1}{|G|} \log |M_G(F)|,$$
where
$$M_G(F):= \prod_{u_1=0}^{n_1-1} \cdots \prod_{u_r=0}^{n_r-1} F(e^{2\pi i u_1/n_1},\ldots , e^{2\pi i u_r/n_r}) \in \mathbb Z. $$
Of course this measure is really defined on $\mathbb Z [x_1,\ldots ,x_r] /\langle x_1^{n_1}-1,\ldots ,x_r^{n_r}-1\rangle$, and writing
$$ F(x_1,\ldots ,x_r) =\sum_{u\in G} a_u x_1^{u_1}\cdots x_r^{u_r}  \text{ mod } \langle x_1^{n_1}-1,\ldots ,x_r^{n_r}-1 \rangle $$
we have 
$$M_G(F)=\mathscr{D}_G(a_{g_1},\ldots ,a_{g_n}). $$

Notice that for a polynomial $f$ in $\mathbb Z [x]$ not vanishing on the unit circle the traditional Mahler  measure
is a limit of Lind measures
$$ m(f) = \lim_{n\rightarrow \infty} m_{\mathbb Z_n} (f). $$

It is not immediately clear how to extend  Lind's original definition to a non-abelian finite group. See  Dasbach and Lal\'{i}n \cite{Lalin} for one approach. Here we suggest that the group determinant formulation  provides  a natural way to do this.

For a  finite group $G$ with generators $\alpha_1,\ldots ,\alpha_r$  and relations $\mathscr{U}(\alpha_1,\ldots,\alpha_r)$,
and  a {\it polynomial} $F$ in $\mathbb Z[x_1,\ldots ,x_r]/\mathscr{U}(x_1,\ldots ,x_r)$ 
$$ F(x_1,\ldots ,x_r) =\sum_{g=\alpha_1^{m_1}\cdots \alpha_r^{m_r}\in G} a_g x_1^{m_1} \cdots x_r^{m_r} \text{ mod } \mathscr{U}(x_1,\ldots ,x_r), $$
we define
$$m_G(F) = \frac{1}{|G|} \log |M_G(F)|,\;\;\;  M_G( F):=\mathscr{ D}_G(a_{g_1},\ldots , a_{g_n}). $$
If multiplication of monomials follows the group relations then, as in \eqref{mult}, we recover the usual multiplicative property 
of Mahler measures
$$ M_G(fh)=M_G(f)M_G(h). $$

For example, for the dihedral group
\begin{align*}  D_{2n} & =\langle \mathcal{R},\mathcal{F} \; | \; \mathcal{R}^n=1, \mathcal{F}^2=1,\mathcal{RF}=\mathcal{FR}^{n-1}\rangle \\ & = \{1,\mathcal{R},\ldots ,\mathcal{R}^{n-1}, \mathcal{F},\mathcal{FR},\ldots ,\mathcal{FR}^{n-1}\}, \end{align*}
we define the measure of a polynomial in  $\mathbb Z[x,y] /\langle x^n-1,y^2-1,xy-yx^{n-1}\rangle$, reduced to the form
\be \label{dipoly}   F(x,y) = \sum_{i=0}^{n-1} a_i x^{i} + \sum_{i=0}^{n-1} b_i \: yx^{i},\ee
to be
$$ M_{D_{2n}}(F) = \mathscr{D}_{D_{2n}}(a_1,\ldots ,a_n,b_1,\ldots ,b_n). $$

\section{The dihedral group}\label{dihedral}

For $G=D_{2n}$ the representations are fairly straightforward, see for example Serre \cite[\S5.3]{Serre}. 
We have the  degree 1 representations, the basic characters  $\chi$ with $\chi (\mathcal{F})=\pm 1$ and $\chi(\mathcal{R})=1$ and  when $n$ is even $\chi(\mathcal{R})=-1$, giving us two or four linear factors as $n$ is odd or even. The other representations all  have degree 2 and arise from
$$  \rho(\mathcal{R}^j)= \left( \begin{matrix} w^j & 0 \\ 0 & w^{-j} \end{matrix} \right),\;\; \rho(\mathcal{F R}^j)= \left( \begin{matrix} 0 & w^{-j} \\ w^j & 0 \end{matrix} \right),     $$
where the $w$ run through the complex  $n$th roots of unity.

Writing $a_i$ for the $x_{\mathcal{R}^i}$ and $b_i$ for the $x_{\mathcal{FR}^i}$ these  lead to 
$$ \sum_g x_g \rho(g) =\left( \begin{matrix}\displaystyle  \sum_{i=0}^{n-1} a_i w^{i} & \displaystyle \sum_{i=0}^{n-1} b_i w^{-i} \\
\displaystyle\sum_{i=0}^{n-1} b_i w^{i} & \displaystyle\sum_{i=0}^{n-1} a_i w^{-i}   \end{matrix} \right),$$
and  the  real quadratic factors, with coefficients in $\mathbb Q(\cos(2\pi/n))$,
$$Q(w):=\left( \sum_{i=0}^{n-1} a_i w^{i}\right)\left(  \sum_{i=0}^{n-1} a_i w^{-i} \right) - \left( \sum_{i=0}^{n-1} b_i w^{i}\right)\left(  \sum_{i=0}^{n-1} b_i w^{-i} \right). $$

Thus with $w_n$ denoting the primitive $n$th root of unity
\be \label{rootofunity} w_n:= e^{2\pi i /n}, \ee
the group determinant factors in $\mathbb C[a_1,\ldots ,a_n,b_1,\ldots ,b_n]$ as follows:

\vspace{2ex}
\noindent
{\bf When $n=2k+1$ is odd}

\noindent
$\mathscr{D}_G(a_0,\ldots ,a_{n-1},b_0,\ldots ,b_{n-1})$ factors as two linear and the square of $k$ quadratics
$$\left(\sum_{i=0}^{n-1} a_i + \sum_{i=0}^{n-1}b_i\right)\left(\sum_{i=0}^{n-1} a_i - \sum_{i=0}^{n-1}b_i\right)\prod_{j=1}^k Q(w_n^j)^2. $$

\vspace{2ex}
\noindent
{\bf When $n=2k$ is even}

\noindent
$\mathscr{D}_G(a_0,\ldots ,a_{n-1},b_0,\ldots ,b_{n-1})$ factors as the product of  four linear factors
$$ \left(\sum_{i=0}^{n-1} a_i + \sum_{i=0}^{n-1}b_i\right)\left(\sum_{i=0}^{n-1} a_i - \sum_{i=0}^{n-1}b_i\right)\left(\sum_{i=0}^{n-1} (-1)^ia_i + \sum_{i=0}^{n-1} (-1)^ib_i\right)\left(\sum_{i=0}^{n-1}(-1)^i a_i - \sum_{i=0}^{n-1}(-1)^ib_i\right) $$
and the square of $k-1$ quadratics
$$ \prod_{j=1}^{k-1} Q(w_n^j)^2. $$
Note that in both cases we can write
$$ \mathscr{D}_G(a_0,\ldots ,a_{n-1},b_0,\ldots ,b_{n-1})=\prod_{j=0}^{n-1} Q(w_n^j). $$


\vspace{1ex}

Writing these expressions in terms of  values at roots of unity thus motivates us to define the Lind-Mahler
measure of a polynomial,
$$F(x,y)=\sum a_{ij} y^ix^j \in \mathbb Z[x,y],$$
or more generally in $\mathbb Z[x,x^{-1},y,y^{-1}],$
relative to the dihedral group  $G=D_{2n}$  to  be
$$ m_G(F)=\frac{1}{2n} \log |M_G(F)|, $$
with
\begin{align*} M_G(F): & =\prod_{j=0}^{n-1} \frac{1}{2}\Big( F(w_n^j,1)F(w_n^{-j},-1)+ F(w_n^{-j},1)F(w_n^{j},-1)\Big),\\
& = M_{\mathbb Z_n} \left( \frac{1}{2} \left(F(x,1)F(x^{-1},-1)+F(x,-1)F(x^{-1},1)\right)\right), 
\end{align*}
where if we want the usual multiplicative property $M_G(fg)=M_G(f)M_G(g)$ we must use the dihedral relationship
on monomials, $x^jy=yx^{n-j}$, when multiplying two  polynomials.

Notice  for an $F(x,y)=f(x)+yg(x)$ with $f(x),g(x)$ in $\mathbb Z [x,x^{-1}]$ we have 
$$ M_{D_{2n}}(F)=M_{\mathbb Z_n}\Big( f(x)f(x^{-1})-g(x)g(x^{-1})\Big). $$
If $f(x)=x^kf(x^{-1})$, $g(x)=x^{k}g(x^{-1})$ for some $k$ we have
$$ M_{D_{2n}}(F) = M_{\mathbb Z_n \times \mathbb Z_2}(F), $$
a reciprocal property similar to  Dasbach and Lal\'{i}n's \cite[Theorem 12]{Lalin}.

We also have
$$ M_{D_{2n}}(f(x)) =M_{D_{2n}}(yf(x))=M_{\mathbb Z_n}(f(x))^2, $$
and so $\lambda (D_{2n})\leq \lambda(\mathbb Z_n)^2$.

\section{Minimal Values of Dihedral Determinants and Measures} \label{values}

In this section we obtain some restrictions on the values taken by a dihedral group determinant with integer variables. These will be enough to determine the minimal non-trivial determinant for any  dihedral group $G=D_{2n}$ of order less than $3.79\times 10^{47}$.

\begin{theorem}\label{main}
$$\lambda(D_{2n})=\begin{cases}  3, & \text{ if $3\nmid n$, }\\
 4, & \text{ if $n=3m$, \; $2\nmid m$, }\\
 5, & \text{ if $n=2\cdot 3m$,\; $5\nmid m$, }\\
 7, & \text{ if $n=2\cdot 3\cdot 5 m$,\; $7\nmid m$, }\\
11, & \text{ if $n=2\cdot 3\cdot 5 \cdot 7 m$,\; $11\nmid m$, }\\
13, & \text{ if $n=2\cdot 3\cdot 5 \cdot 7 \cdot 11 m$,\; $13\nmid m$, }\\
16, & \text{ if $n=2\cdot 3\cdot 5 \cdot 7 \cdot 11\cdot 13  m$,\; $2\nmid m$, }\\
17, & \text{ if $n=2^2\cdot 3\cdot 5 \cdot 7 \cdot 11\cdot 13  m$,\; $17\nmid m$, }\\
19, & \text{ if $n=2^2\cdot 3\cdot 5 \cdot 7 \cdot 11\cdot 13 \cdot 17 m$,\; $19\nmid m$, }\\
23, & \text{ if $n=2^2\cdot 3\cdot 5 \cdot 7 \cdot 11\cdot 13 \cdot 17 \cdot 19m$,\; $23\nmid m$, }\\
27, & \text{ if $n=2^2\cdot 3\cdot 5 \cdot 7 \cdot 11\cdot 13 \cdot 17 \cdot 19\cdot 23m$,\; $3\nmid m$, }\\
29, & \text{ if $n=2^2\cdot 3^2\cdot 5 \cdot 7 \cdot 11\cdot 13 \cdot 17 \cdot 19\cdot 23m$,\; $29\nmid m$, }\\
\end{cases}$$
and  for primes $p=31$ to $113$
$$ \lambda (D_{2n}) = p,\;\;\;  \text{ if $n=2^2\cdot 3^2\cdot \left( \prod_{\substack{ \text{ all primes  }\\ 5\leq q <p}}q\right)\cdot m$,\;\;  $p \nmid m$.} $$
\end{theorem}
This just leaves the groups $D_{2n}$ when $n$ is a multiple of 
$$N=2^2\cdot 3^2\cdot 5 \cdot 7 \cdot 11 \cdot \; \cdots \; \cdot 107\cdot 109\cdot  113,$$ 
where $\lambda(D_{2N})=125$ or $127$.  Some multiples of $N$ can be determined. For example when
 $n$ a multiple of $5N$ we can rule out 125 and  for $p=127$ to $241$ 
$$ \lambda (D_{2n})=p, \;\;\; \text{ if $n=2^2\cdot 3^2\cdot 5^2\cdot \left( \prod_{7\leq q <p} q \right) \cdot m$,  $\;\; p\nmid m$,} $$
with $\lambda( D_{2n})=243$ or $251$  when $n=2^2\cdot 3^2 \cdot 5^2 \cdot 7\cdots 241$. For  $n$ a multiple of $15N$ 
$$\lambda (D_{2n})=\begin{cases}251, &  \text{ if $n=2^2\cdot 3^3 \cdot 5^2 \cdot \left( \prod_{7\leq q \leq 241} q \right) \cdot m$,  $\;\; 251 \nmid m$,}\\
 256, & \text{ if $n=2^2\cdot 3^3 \cdot 5^2 \cdot \left( \prod_{7\leq q \leq 251} q \right) \cdot m$,  $\;\; 2 \nmid m$,} \\
p, &  \text{ if $n=2^3\cdot 3^3 \cdot 5^2 \cdot \left( \prod_{7\leq q <p} q \right) \cdot m$,  $\;\; p \nmid m$,} \end{cases}$$
for primes $p$ from  $257$ to $337$, with  $\lambda( D_{2n})=343$ or $347$  when $n=2^3\cdot 3^3 \cdot 5^2 \cdot 7\cdots 337.$ 

\vspace{1ex}
Theorem \ref{main} will follow immediately from two lemmas.
Laquer \cite{Laquer} and Newman \cite{Newman1}  proved that $\mathscr{D}_{\mathbb Z_n}(a_0,\ldots ,a_{n-1})$ achieves all integers coprime to $n$. See Mahoney and Newman \cite{Mahoney} for other abelian groups.
We similarly show for $G=D_{2n}$ that any positive integer coprime to $2n$ will be the absolute value of an integer  determinant. In particular we can achieve any odd  prime $p\nmid n$:

\begin{lemma}\label{achieve}
Let $G=D_{2n}$. 

If $m=2t+1$ with $(m,n)=1$ and $0\leq t<n$ then 
$$|\mathscr{D}_G(\overbrace{1,\ldots ,1}^{t+1},\overbrace{0,\ldots ,0}^{n-1-t},\overbrace{1,\ldots ,1}^{t},\overbrace{0,\ldots ,0}^{n-t})|=m. $$
 For $t\geq n$ we can still achieve $m$ by wrapping around the remaining 1's;
$$a_i=|\{ 0\leq j\leq t\;: j\equiv i \text{ mod } n\}|,\;\;\; b_i=|\{ 0\leq j\leq t-1\;: j\equiv i \text{ mod } n\}|.$$

If $2\nmid n$ then $\mathscr{D}_{G}(1,1,0,\ldots ,0)=2^2$.

If $2 || n $ then  $\mathscr{D}_{G}(1,0,1,0,\ldots ,0)=2^4$.

If $2^2 || n $ then  $\mathscr{D}_{G}(1,0,0,0,1,0,\ldots ,0)=2^8$.

If $3||n$  then  $|\mathscr{D}_{G}(\underbrace{1,0,0,1,0,\ldots ,0}_n\underbrace{1,0,\ldots , 0}_n)|=3^3$.

More generally if $p$ is odd and $p^{\alpha}|| n$ then we can achieve $p^{p^{\alpha}}$ as the absolute value of an integer group determinant,
if $2^{\alpha}||n$ then we can achieve $2^{2^{\alpha+1}}$.

\end{lemma}

Without absolute values we  show:

\begin{theorem} \label{coprime2n}
Let $G=D_{2n}$ and $m$  be an  integer coprime to $2n$.

If $n$ is odd then $m$ is a $\mathscr{D}_G(a_1,...,a_n,b_1,...,b_n)$ for some integers $a_1,...,a_n,b_1,...,b_n$.

If $n$ is even then either $m$ or $-m$ is a determinant, whichever is $1$ mod $4$.

If $n$ is even the odd determinants are all $1$ mod 4.

\end{theorem}

Next, we obtain a  condition on divisibility by  primes dividing $2n$:

\begin{lemma} \label{div}
Let $G=D_{2n}$ and $a_0,\ldots ,a_{n-1},b_0,\ldots ,b_{n-1}$  be in $\mathbb Z$.

\noindent
If $p^{\alpha}||n$ and $p\mid \mathscr{D}(a_0,...,a_{n-1},b_0,...,b_{n-1})$ then $p^{2\alpha+1}\mid \mathscr{D}(a_0,...,a_{n-1},b_0,...,b_{n-1})$.

\noindent
If $2\nmid n$ and $2\mid \mathscr{D}(a_0,...,a_{n-1},b_0,...,b_{n-1})$ then $4\mid \mathscr{D}(a_0,...,a_{n-1},b_0,...,b_{n-1})$.

\noindent
If $2||n$ and $2\mid \mathscr{D}(a_0,...,a_{n-1},b_0,...,b_{n-1})$ then $2^4 ||  \mathscr{D}(a_0,...,a_{n-1},b_0,...,b_{n-1})$ or  $2^6\mid \mathscr{D}(a_0,...,a_{n-1},b_0,...,b_{n-1})$.


\noindent
If $2^{\alpha}||n$, $\alpha\geq 2$,  and $2\mid \mathscr{D}(a_0,...,a_{n-1},b_0,...,b_{n-1})$ then $2^{2\alpha+4}\mid \mathscr{D}(a_0,...,a_{n-1},b_0,...,b_{n-1})$.

\end{lemma}

We observe that
\begin{align*}  & M_{D_{2p}}\Big( \frac{x^{(p+1)/2}-1}{x-1} +(x-1)+ y \left( \frac{x^{(p-1)/2}-1}{x-1} +x^{-1}(x-1)\right)\Big) \\ &  = M_{\mathbb Z_p}\left( x^{-(p-1)/2} \Phi_p(x) +x^{-1}(x-1)^2\right)= p^3,
\end{align*}
so extra conditions would be needed on $n$ to rule out $p^3$ when $p||n$. A similar problem occurs in the cyclic case where
$$ M_{\mathbb Z_p}\left( \Phi_p(x)+(1-x)\right)=p^2. $$

\section{ The Dihedral Groups of Order $2p$ and $4p$ }   \label{D2p}

So far we have concentrated on  just finding the smallest non-trivial group determinant or measure.
Laquer \cite{Laquer} and Newman \cite{Newman1} obtained a complete description of the group determinants
for $G=\mathbb Z_p$, showing that one achieves anything of the form $p^a m$, $p\nmid m$ with $a=0$ or $a\geq 2$.
Here we similarly give a complete description of the measures for 
the dihedral group $D_{2p}$:

\begin{theorem} \label{D_2p}
Suppose that $G=D_{2p}$ with $p$ an odd prime. Then the values achieved as integer  group determinants take the form
$2^a p^b m$ for any integer $m$ with $(m,2p)=1$,  $a=0$ or $a\geq 2$ and $b=0$ or $b\geq 3$. 

\end{theorem}

 For $G=\mathbb Z_{2p}$ with $p$ an odd prime Laquer \cite{Laquer}  showed that the values take the form $2^a p^bm$ 
for any $(m,2p)=1$ and $a=0$ or $a\geq 2$ and $b=0$ or $b\geq 2$. Likewise we can give a complete description for $D_{4p}$.

\begin{theorem} \label{D_4p}
Suppose that $G=D_{4p}$ with $p$ an odd prime.

 The odd values achieved as integer  group determinants are the $m\equiv 1$ mod 4 with $p\nmid m$ or $p^3\mid m$.

The even values are the  integers of  the form
$2^a p^b m$ for any integer $m$ with $(m,2p)=1$,  $a=4$ or $a\geq 6$ and $b=0$ or $b\geq 3$. 

\end{theorem}

For $G=\mathbb Z_{2^k}$  the odd integers are all  determinants. The  even determinant values are $\mathscr{E}=4\mathbb Z$, $16\mathbb Z$ and $32\mathbb Z$ when $k=1,2$ or 3, but  only upper and lower set inclusions $2^{2k-1}\mathbb Z \subseteq \mathscr{E}\subseteq 2^{k+2}\mathbb Z$ are known  for $k\geq 4$, see Kaiblinger \cite[Theorem 1.1]{Norbert2}. Similarly for the dihedral groups $D_{2^k}$ we give a complete description of the determinants  for $k\leq 4,$ and upper and lower containments for $k\geq 5$.
\begin{theorem} \label{D_8}
 For $G=\mathbb Z_2 \times \mathbb Z_2$ the group determinant with integer variables are
$$ 4m+1 \;\;\; \hbox{ and }\;\;\; 2^4(2m+1) \;\;\; \hbox{ and } \;\;\; 2^6 m,\;\;\; m\in \mathbb Z. $$

For  $G=D_8$ the  determinant takes the values $4m+1$ and $2^8 m$, $m$ in $\mathbb Z$.

For  $G=D_{16}$ the  determinant takes the values $4m+1$ and $2^{10} m$, $m$ in $\mathbb Z$.

For $G=D_{2^k}$ with $k\geq 4$ the odd values achieved are the integers 1 mod 4.  The set $\mathscr{E}$ of even values achieved satisfies
$$ 2^{3k}\mathbb Z \subseteq \mathscr{E} \subseteq 2^{2k+2}\mathbb Z,$$
and contains  values with $2^{2k+2} || M_G(F)$.

\end{theorem}

For $G=\mathbb Z_{p^k}$ with $p$ odd and $k\geq 2$ the values coprime to $p$ are all achievable while
the set of values that are multiples of $p$ were shown by Newman \cite{Newman2} to satisfy
$$ p^{2k}\mathbb Z \subsetneq \mathscr{P} \subseteq p^{k+1}\mathbb Z $$
with $\mathscr{P}=27\mathbb Z$ when  $k=2$  and $p=3$ but  $\mathscr{P}\neq p^{k+1}\mathbb Z$  for $p\geq 5$.

Similarly for $G=D_{2p^k}$ we can obtain a complete description for a few primes when $k=2$, and in general upper and lower set inclusions.

\begin{theorem}\label{D_18} For $G=D_{2p^2}$ with $p=3,5$ or 7 the measures take the form $2^a p^b m$, $(m,2p)=1$ with $a=0$ or $a\geq 2$ and $b=0$ or $b\geq 5$.

 In general,  if $G=D_{2p^k}$ with $p$ an odd prime and $k\geq 2$ the measures must take the form
$2^ap^bm$, $(m,2p)=1,$ with $a=0$ or $a\geq 2$ and $b=0$ or $b\geq 2k+1$. We can achieve everything of this form with $b=0$ or $b\geq 3k$. There are measures with $b=2k+1$.

\end{theorem}

\section{Proofs}

Recall that
$$   x^n-1 =\prod_{m\mid n} \Phi_m(x),\;\;\;\;    \Phi_m(x):=\prod_{\stackrel{j=1}{\gcd(j,m)=1}}^{m} (x- e^{2\pi i j/m}),$$
where $\Phi_m(x)$ is  the $m$th cyclotomic polynomial,
an  irreducible polynomial in $\mathbb Z[x]$ whose roots are the primitive $m$th roots of unity.
We write $\text{Res}(F,G)$ for the resultant of $F$ and $G$. We shall need
the resultant of two cyclotomic polynomials, see for example  Apostol \cite{Apostol}, though this dates  back to  E. Lehmer \cite{ELehmer}.
For $m>n$
$$ \left|\text{Res}\left(\Phi_n,\Phi_m\right)\right| = \begin{cases} p^{\phi (n)}, &
\text{if $m=np^{\alpha},$} \\ 1, & \text{ else.} \end{cases}$$

\begin{proof}[Proof of Lemma \ref{achieve}] We begin with powers of  2.
If we set $y_1=\cdots =y_n=0$ then 
\begin{align*}  M_{D_{2n}}(a_0 +a_1 x+\cdots a_{n-1}x^{n-1})  & = M_{\mathbb Z_{n}}(a_0 +a_1 x+\cdots a_{n-1}x^{n-1})^2 \\
 & = |\text{Res}(a_0 +a_1 x+\cdots a_{n-1}x^{n-1},x^n-1)|^2. 
\end{align*}
Observe that if $2\nmid n$ we have $|\text{Res}(x+1,x^n-1)|=|\text{Res}(x+1,x-1)|=2$, and if $2^{\alpha}|| n$
we have 
$$ |\text{Res}(x^{2^{\alpha}}+1,x^n-1)|=\prod_{j=0}^{\alpha}|\text{Res}( \Phi_{2^{\alpha+1}},\Phi_{2^j})|=\prod_{j=0}^{\alpha} 2^{\phi (2^j)} =2^{2^{\alpha}}.$$
Taking these polynomials and reducing mod $x^n$ as necessary we get the squares of  $2$ if $2\nmid n$ and $2^{2^\alpha}$
if $2^{\alpha}||n$. For odd $p$ we could similarly get $p^{2p^{\alpha}}$ when $p^{\alpha}|| n$.

Instead suppose that $m=2t+1$ with $(m,n)=1$ and consider the polynomial
$$ F(x,y)= \sum_{i=0}^t x^i + \sum_{i=0}^{t-1} yx^i = f(x) + yg(x),  \;\; f(x)= \frac{x^{t+1}-1}{x-1},\;\; g(x)=\frac{x^t-1}{x-1}$$
where we reduce $F$  mod $x^n-1$ to make it of the form \eqref{dipoly}  if $t\geq n$.
 Then
$$ M_G(F(x,y))=M_{\mathbb Z_n}\left( f(x)f(x^{-1})-g(x)g(x^{-1}) \right) = M_{\mathbb Z_n}\left( x^{-t} \left(\frac{x^m-1}{x-1}\right) \right).$$
Since $(n,m)=1$ the squared quadratic terms assemble to give
$$  \prod_{k\mid n, k\neq 1,2}\prod_{d\mid m, d\neq 1} |\text{Res}(\Phi_d, \Phi_k)|=1. $$
For $n$ odd the two linear terms give us $f(1)^2-g(1)^2=m$ and $M_G(F)=m$. For $n$ even  we also have $f(-1)^2-g(-1)^2 = (-1)^t$ and $M_G(F)=m$ if $m\equiv 1$ mod 4 and 
$M_G(F)=-m$ if $m\equiv 3$ mod 4.

If $p=2t+1$ and $p^{\alpha}||n$ we consider
$$ F= \sum_{i=0}^t x^{ip^{\alpha}} + \sum_{i=0}^{t-1} yx^{i p^{\alpha}}=f(x^{p^{\alpha}})+yg(x^{p^{\alpha}}), $$
reducing mod $x^n-1$ if necessary. We similarly obtain 
$$ M_G(F) =  M_{\mathbb Z_n}\left( x^{-tp^{\alpha}} \Phi_{p^{\alpha +1}}\right), $$
where $|\text{Res}(\Phi_d,\Phi_{p^{\alpha+1}})|=1$ for the $d\mid n$, except for the $d=p^j$ which contribute
$$|M_G(F)|=  |\text{Res}(x-1,\Phi_{p^{\alpha+1}})|\prod_{j=1}^{\alpha} |\text{Res}(\Phi_p^j,\Phi_{p^{\alpha+1}})|=p\prod_{j=1}^{\alpha} p^{\phi(p^j)}= p^{p^{\alpha}}. $$

\end{proof}

\begin{proof}[Proof of Theorem \ref{coprime2n}] From the proof of Lemma \ref{achieve} we can obtain 
any $(m,2n)=1$ when $n$ is odd and those with $m\equiv 1$ mod 4 when $n$ is even. Note switching the roles of $f(x)$ and $g(x)$ switches the sign of $M_G(f(x)+yg(x))$ when $n$ is odd but not for $n$ even. So it just remains to show 
the odd $D_{2n}$  measure of an $F(x,y)=f(x)+yg(x)$  must be $1$ mod 4 when $n$ is even. The square terms will of course be 1 mod 4, so that  leaves the contribution from the four  linear terms  $f(1)^2-g(1)^2$ and $f(-1)^2-g(-1)^2$. Since $f(-1)\equiv f(1)$ mod $2$ we have $f(-1)^2\equiv f(1)^2$ mod 4 and $((f(1)^2-g(1)^2)(f(-1)^2-g(-1)^2)\equiv (f(1)^2-g(1)^2)^2\equiv 1$ mod 4. 
\end{proof}

\begin{proof}[Proof of Lemma \ref{div}] We write $\mathscr{D}:=\mathscr{D}_G(a_0,\ldots ,a_{n-1},b_0,\ldots ,b_{n-1})$.

Suppose  first that $p$ is odd with $p^{\alpha}||n$.  If $\alpha=0$ then there is nothing to show, so assume that $\alpha\geq 1$.  If $n=mp^{\alpha}$ then, writing $sm+tp^{\alpha}=1$ and observing that $w_n=w_{p^{\alpha}}^sw_m^t$, and that $Q(w)=Q(w^{-1}),$ we have
$$ \prod_{j=1}^{[(n-1)/2]} Q(w_n^j)^2= \prod_{j=0}^{p^{\alpha}-1} \prod_{\substack{u=0\\w_m^uw_{p^{\alpha}}^j\neq\pm 1} }^{m-1} Q(w_m^u w_{p^{\alpha}}^j) =A_1^ 2\prod_{\substack{d\mid m\\d\neq 1,2} }A_d^2, $$
where, splitting into the order of the $p^{\alpha}$th and $m$th roots of unity, and re-pairing the $d$th roots with their conjugates, we get,
$$ A_1:= \prod_{v=1}^{\alpha} A_1(v),\;\; A_d:= A_d(0)\prod_{v=1}^{\alpha} A_d(v), $$
where
$$  A_1(v):=\prod_{\substack{j=1\\p\nmid j}}^{(p^v-1)/2} Q(w_{p^v}^j) \text{  if $m$ is  odd }, \;\;\;  A_1(v):=\prod_{\substack{j=1\\p\nmid j}}^{(p^v-1)/2} Q(w_{p^v}^j) Q(-w_{p^v}^j)\text{ if  $m$ is even }, $$
and
$$ A_d(0):=\prod_{\substack{t=1\\(t,d)=1}}^{d/2} Q(w_d^t),\;\; A_d(v):=\prod_{\substack{j=1\\p\nmid j}}^{p^v-1}\prod_{\substack{t=1\\(t,d)=1}}^{d/2} Q(w_d^tw_{p^v}^j). $$
Note, since $Q(w)=Q(w^{-1}),$ our products are over  complete sets of conjugates and the $A_j(i)$ are all integers.
Notice also that our linear factors equal
$$ A_0 :=\begin{cases} Q(1), & \text{ if $m$ is odd,} \\ Q(1)Q(-1), & \text{ if $m$ is even,}\end{cases}$$ 
and
$$ \mathscr{D}=A_0 A_1^2 \prod_{\substack{d\mid m\\d\neq 1,2}} A_d^2.$$ 

Setting $\pi_{\alpha}=1-w_{p^{\alpha}}$, observe  that the $p$-adic absolute value on $\mathbb Q$ extended to $\mathbb Q(w_n)$ 
has $|\pi_{\alpha}|_p= p^{-1/\phi(p^{\alpha})} < 1$. Since all the $w_{p^{v}}^j\equiv 1$ mod $\pi_{\alpha}$ in $\mathbb Z [w_n]$, we gain the following congruences mod $\pi_{\alpha}$ in $\mathbb Z [w_n]$ and, since all the terms are integers, mod $p$ in $\mathbb Z$:
$$ A_d(v)\equiv A_d(0)^{\phi(p^v)} \text{ mod } p,\;\;\; A_1(v) \equiv A_0^{\phi (p^v)/2} \text{ mod } p.$$
Hence if 
$p\mid \mathscr{D}$ 
then either we have $p\mid A_d(v)$ for some $d$ and $0\leq v\leq \alpha, $ and hence $p\mid A_d(v)$ all $0\leq v\leq \alpha$,  giving  $p^{\alpha+1}\mid A_d$ and $p^{2\alpha+2}\mid  \mathscr{D}$, or $p\mid A_0$ or $p\mid A_1(v)$ for some $1\leq v \leq \alpha$ and $p$ divides all of these, giving $p^{\alpha}\mid A_1$, $p\mid A_0$  and $p^{2\alpha+1} \mid \mathscr{D}. $

Suppose  first that $2\nmid n$. Notice that the two  linear factors are congruent 
mod 2 so that their product is either odd or a multiple of 4, while the integer from the quadratic factors is squared. Thus $2\mid  \mathscr{D}$ implies that $2^2\mid \mathscr{D}$.

So suppose that $2^{\alpha}||n$ with $\alpha \geq 1$ and $2\mid \mathscr{D}$. Similar to the $p$ odd case
we write $n=2^{\alpha}m$ and end up with 
$$ \mathscr{D}=Q(1)Q(-1) A_1^2 \prod_{\substack{d\mid m\\d\neq 1}}A_d^2, $$
with $A_d$ as before and
$$ A_1=\prod_{v=2}^{\alpha} A_1(v),\;\;\; A_1(v):=\prod_{\substack{j=1\\ j \text{ odd}}}^{2^{v-1}} Q(w_{2^v}^j). $$
Again from 2-adic considerations we have 
$$ A_1(v)\equiv Q(1)^{2^{v-2}} \text{ mod } 2,\;\;\;\; \;\; A_d(v) \equiv A_d(0) \text{ mod } 2. $$
Hence if $2\mid A_d(v)$ for some $d$ and $0\leq v\leq \alpha$ then 2 divides them all and $2^{\alpha+1}\mid A_d$ and $2^{2\alpha+2}\mid \mathscr{D}$. Likewise, if $2\mid Q(1),Q(-1)$ or $A_1(v)$ for some $2\leq v\leq \alpha$
then 2 divides all of them, so  $2^{\alpha -1}\mid A_1$ and, since the four linear factors are all congruent mod 2, also $2^4\mid Q(1)Q(-1)$, giving $2^{2\alpha+2}\mid \mathscr{D}$. Moreover, writing $Q(1)=f(1)^2-g(1)^2$
and $Q(-1)=f(-1)^2-g(-1)^2$, if $Q(1)$ is even then $f(1),g(1),f(-1),g(-1)$ all have the same parity. If all are odd then $2^3\mid Q(1),Q(-1)$ and $2^6\mid Q(1)Q(-1)$. If all are even then $(f(1)/2)^2- (g(1)/2)^2$ 
and $(f(-1)/2)^2- (g(-1)/2)^2$ are either
odd or $0$ mod 4 and $2^4||Q(1)Q(-1)$ or $2^6|Q(1)Q(-1)$.

For $\alpha \geq 2$  we can extract  two further 2's. If  $2^2 \mid A_d(v)$ for some $0\leq v \leq  \alpha$ or $2^2\mid A_1(v)$ for some $2\leq v\leq \alpha$ then, since these terms are squared, we gain two extra 2's. 
This leaves us to consider the cases $2 || A_d(2)$ or $2 || A_1(2)$.

Suppose first that $2 || A_d(2).$  Suppose that $f(x)=\sum_{j=0}^{n-1}a_jx^j$, $g(x)=\sum_{j=0}^{n-1}b_jx^j$  and write
$$ H(x) = \prod_{\substack{j=1\\(j,d)=1}}^{d}Q(xw_d^j),\;\;\; Q(x)=f(x)f(x^{-1})-g(x)g(x^{-1}). $$
 Since we run over a full set of conjugates $H(x)\in \mathbb Z [x,x^{-1}]$. Moreover $H(x^{-1})=H(x)$ 
and, writing $x^j+x^{-j}$ as a  polynomial in $x+x^{-1}, $ we have $H(x)\in \mathbb Z [ x+x^{-1}]$. 
$$ H(x) = \sum_{j=0}^J  B_j (x+x^{-1})^j. $$
Note that
\begin{align*}   A_d(0)^2 & = H(1)\equiv B_0+2B_1 +4B_2 \text{ mod 8}, \\
 A_d(1)^2  &=H(-1)\equiv B_0-2B_1+4B_2 \text{ mod } 8,\\
 A_d(2) & =H(i)=B_0. 
\end{align*}
Hence if $2|| A_d(2)$ we have 
$$A_d(0)^2+A_d(1)^2 \equiv 2B_0\equiv 4 \text{ mod } 8, $$
ruling out  having $2 || A_d(0)$ and  $A_d(1)$. Thus $2^2\mid A_d(0)$ or $A_d(1)$ and $2^4\mid A_d(0)A_d(1)A_d(2)$ and $2^{\alpha+2}\mid A_d$ and $2^{2\alpha+4}\mid \mathscr{D}.$

Similarly if $2 || A_1(2)$  we have
$$ Q(1)Q(-1) A_1(2)^2 = Q(1)Q(-1)Q(i)^2$$
and by the same argument, with $Q(x)$ in place of $H(x),$ we have
$$ Q(1)+Q(-1)\equiv 2Q(i) \text{ mod } 8 $$ 
with $4\mid Q(\pm 1)$, and if $2 || Q(i)$ we can't have $2^2||Q(\pm 1)$. That is, for either  $\epsilon=1$ or $-1$ we must have $8 \mid Q(\epsilon)=f(\epsilon)^2-g(\epsilon)^2$  and $2^2 || Q(-\epsilon)$. Note $f(\epsilon)$ and $g(\epsilon)$ must be the same parity and can't both be odd, else $f(-\epsilon)$ and $g(-\epsilon)$ would both be odd and $8\mid Q(-\epsilon)$. Hence we can write $f(\epsilon)=2\alpha_0$ and $g(\epsilon)=2\beta_0$ and $Q(\epsilon)=4(\alpha_0^2-\beta_0^2)$. Since $2\mid \alpha_0^2-\beta_0^2$ we must have $4\mid \alpha_0^2-\beta_0^2$ and $2^8\mid Q(1)Q(-1)Q(i)^2,$ and again  $2^{2\alpha+4}\mid \mathscr{D}.$

\end{proof}

\begin{proof}[Proof of Theorem \ref{main}]

For $n$ of each given form Lemma \ref{achieve} says that we can achieve the value claimed as $\lambda(D_{2n}),$ while Lemma \ref{div} rules out 
anything smaller.

\end{proof}

\begin{proof}[Proof of Theorem \ref{D_2p}]
Suppose that $G=D_{2p^k}$ with $p$ an odd  prime and $k\geq 1$.
 From Theorem \ref{coprime2n}  we can obtain all values coprime to $2p$. Products of
\begin{align*}
M_G( x+1)  & =M_{\mathbb Z_{p^k}}(x+1)^2=2^2, \\
M_G( (x^2+x+1) +y)  & =M_{\mathbb Z_{p^k}}\left( x^{-2}(x+1)^2(x^2+1)\right)=2^3, 
\end{align*}
give any $2^a$  with $a\geq 2$, and for $\ell\geq k$
\begin{align*}
& M_G\left(  \frac{x^{\frac{(p^{\ell}+1)}{2}}-1}{x-1} + (x-1) + y \left(  \frac{x^{\frac{(p^{\ell}-1)}{2}}-1}{x-1} +x^{-1} (x-1) \right)\right)   \\
 & = M_{\mathbb Z_{p^k}} \left( x^{-\left(\frac{p^{\ell}-1}{2}\right)}\left(\frac{x^{p^{\ell}}-1}{x-1}\right) +x^{-1}(x-1)^2\right) = p^{\ell+2k}
\end{align*}
gives any power $p^b$ with $b\geq 3k$. Products of these achieve anything of the form $2^a p^b m$ with $(m,2b)=1$, $a=0$ or $a\geq 2$ and $b=0$ or $b\geq 3k$.  Lemma \ref{div} shows that 
measures must be of this form with $a=0$ or $a\geq 2$, and  $b=0$ or $b\geq 2k+1$. For $k=1$ these coincide and we have a complete description of the measures.

\end{proof}

\begin{proof}[Proof of Theorem \ref{D_4p}]
 From  Lemma \ref{div} and Theorem \ref{coprime2n}  we know that the values must be of the stated  form and that we can obtain all values coprime to $2p$ that are $1$ mod 4. To deal with the powers of $p$ we show that for 
any $k\geq 1$ there is an $F(x,y)=f(x)+yg(x)$ with
$$ M_G(F)=\delta p^{k+2}, \;\;\; \delta :=\begin{cases} 1, & \text{ if $p^k\equiv 1$ mod 4,}\\
-1, & \text{ if $p^k\equiv -1$ mod 4.}\end{cases}$$
Taking
$$ A=  \frac{1}{2}(p^k+\delta),\;\; B=\frac{1}{4}(p^k-\delta), \;\;\; f(x)=\frac{x^A-1}{x-1},\;\; g(x)=(x^p+1)\left(\frac{x^B-1}{x-1}\right), $$
we have $M_G(F)=M_{\mathbb Z_{2p}}(H(x))$ where
$ H(x)=f(x)f(x^{-1})-g(x)g(x^{-1}). $
Plainly 
$$H(1)=A^2-(2B)^2=\delta p^k. $$
If $x^{p}=-1$ we have $H(x)=x^{-(A-1)}\left( (x^A-1)/(x-1)\right)^2$ and, since $A$ is odd and coprime to $p$, we have  
$$H(-1)=1,\;\;\; |\Res(H(x),\Phi_{2p})|=1. $$

If $x^{p}=1$, $x\neq 1$, we have $H(x)=K(x)/(x-1)(x^{-1}-1)$ where
\begin{align*} K(x^4)  & =(x^{4A}-1)(x^{-4A}-1)-4(x^{4B}-1)(x^{-4B}-1) \\
 & =(x^{2\delta}-1)(x^{-2\delta}-1)-4(x^{-\delta}-1)(x^{\delta}-1) \text{ mod } (x^p-1) \\  & =-(x^{\delta}-1)^2(x^{-\delta}-1)^2. 
\end{align*}
Since $\text{Res}(K(x),\Phi_p)= \text{Res}(K(x^4),\Phi_p)$ and $\text{Res}(x-1,\Phi_p)=p$ we see
that 
$\text{Res}(H(x),\Phi_p)=p^2$  and $M_G(F)=\delta p^{k+2}$.

To obtain the necessary  powers of $2$ we have 
\begin{align*}
& M_G(x^2+1)  = 2^4,\\
& M_G( (x^p+1)+y(x-1))  =\prod_{x^p=-1}-|x-1|^2\prod_{x^p=1} (x+1)(x^{-1}+1)= -2^4,\\
&M_G(1+x^2+x(1+x^p)) = \prod_{x^p=-1}|x^2+1|^2\prod_{x^p=1} |x+1|^4=2^6,\\
& M_G( 1-x^{p+2} +y(x^p+1)(x+1))  =\prod_{x^p=-1}|x^2+1|^2\prod_{x^p=1} -(x+1)^2(x^{-1}+1)^2 = -2^6.
\end{align*}
To add an additional $\pm 2^l$  to the $2^6$ we choose $m=1$ or 3 so that $mp\pm 2^l\equiv 1$ mod 4,
set $mp\pm 2^l=2t+1$, where $t$ is even,  and take $F(x,y)=f(x)+yg(x)$ with
\begin{align*}
f(x) & =  \left( 1+x^2 + x(1+x^p)\right) \left( \frac{x^{t+1}-1}{x-1}\right) -m \left(\frac{x^{2p}-1}{x-1}\right), \\
g(x) & =  \left( 1+x^2 + x(1+x^p)\right) \left( \frac{x^{t}-1}{x-1}\right) -m \left(\frac{x^{2p}-1}{x-1}\right).
\end{align*}
Then $H(x)=f(x)f(x^{-1})-g(x)g(x^{-1})$ has
$$ H(1)=    (4(t+1)-2mp)^2-(4t-2mp)^2=2^4(2t+1-mp) =\pm 2^{l+4}. $$
Since
$$(x^{t+1}-1)(x^{-(t+1)}-1)-(x^t-1)(x^{-t}-1)= x^{-t} (x^{2t+1}-1)(x^{-1}-1), $$
for $x^p=-1$ we have
$$ H(x)= |x^2+1|^2x^{-t}\left( \frac{x^{2t+1}-1}{x-1}\right), \;\; \; H(-1)=2^2,\;\; \;\text{Res}(H(x),\Phi_{2p})=1. $$
 For $x^p=1$ with $x\neq 1$, 
$$ H(x)=|1+x|^4x^{-t}\left(\frac{x^{2t+1}-1}{x-1}\right),\;\;\; \text{Res}(H(x),\Phi_{p})=1.$$ 
Hence $M_G(F)=\pm 2^{l+6}$. Products then achieve all the stated forms.

\end{proof}

\begin{proof}[Proof of Theorem \ref{D_8}] Suppose that $G=D_{2^k}$ with $k\geq 2$.  From   Theorem \ref{coprime2n}  we know that the odd values taken are exactly the integers 1 mod 4.

Corresponding to  $k=2$ we have $G=\mathbb Z_2 \times \mathbb Z_2$. From Lemma \ref{div} the even measures satisfy $2^4 ||M_G(F)$ or $2^6\mid M_G(F)$. Writing $F(x,y)=f(x)+yg(x)$ for we have $M_G(F)=(f(1)^2-g(1)^2)(f(-1)^2-g(-1)^2)$  and the measures are readily achieved with
\begin{align*}
&M_G( 1+m(x+1)+ym(x+1))  =1+4m,\\
& M_G( 2+ m(x+1) + ym(x+1))  = 2^4(2m+1),\\
&M_G( 3+ (m-1)(x+1) + y( 1+(m-1)(x+1)))  = 2^6m.
\end{align*}

For $G=D_{2^k}$ with $k\geq 3$ the upper bound on the even values  $\mathscr{E}\subseteq 2^{2k+2}\mathbb Z$ follows from Lemma \ref{div}. For  the lower 
bound $2^{3k}\mathbb Z \subseteq \mathscr{E}$  take
$$ f(x)=2 + m\left(\frac{x^{2^{k-1}}-1}{x-1}\right) ,\;\;\; g(x)=(x+1) - m\left(\frac{x^{2^{k-1}}-1}{x-1}\right).  $$
Writing $H(x)=f(x)f(x^{-1})-g(x)g(x^{-1})$, we have $H(1)=2^{k+2}m$,   with $H(x)=(1-x)(1-x^{-1})$ when  $x^{2^{k-1}}=1$, $x\neq 1$,  and
$$ M_G\left( f(x)+yg(x) \right)= M_{\mathbb Z_{2^{k-1}}}(H(x))=2^{3k}m. $$
It is not hard to see that $2^{2k+2}||M_G(2+(1-x))$.

For $G=D_8$ this gives  $2^9\: \mathbb Z \subseteq \mathscr{E} \subseteq 2^8\: \mathbb Z.$ 
For the missing   multiples of $2^8$:
\begin{align*} 
& M_G\left( 2+ k\frac{x^4-1}{x-1} + yk\frac{x^4-1}{x-1}\right)=2^8(4k+1),\\
 & M_G\left( (x^2+1)(x-1)+ k\frac{x^4-1}{x-1} + y\left( (x+1)+k\frac{x^4-1}{x-1}\right)\right)=-2^8(4k+1).
\end{align*}

For $G=D_{16}$ we have  $2^{12}\: \mathbb Z \subseteq \mathscr{E} \subseteq 2^{10}\: \mathbb Z.$
The remaining multiples of $2^{10}$ are readily  obtainable using:
\begin{align*}
& M_G\left(       (1+x^2)(1+x^4) -(1-x) \right)=2^{10},  \\
& M_G\left( (1+x^2)(1+x^4) +(x-1)(x^2+1)+y(x-1)\right)=-2^{10}, \\
& M_G\left( (1+x^2) - \left(\frac{x^8-1}{x-1}\right) + y(x+1)\right) = 2^{11},\\
& M_G\left( (1+x^2) +y\left( (x+1)-\left(\frac{x^8-1}{x-1}\right) \right) \right) = -2^{11}.
\end{align*}

\end{proof}

\begin{proof}[Proof of Theorem \ref{D_18}] From the proof of Theorem \ref{D_2p} we know that the measures
are of the form $2^a p^b m$,  $(m,2p)=1$ with $a=0$ or $a\geq 2$ and $b=0$ or $b\geq 2k+1$ 
where we can achieve anything of this type with $b=0$ or $b\geq 3k$.
For $p\nmid AB$ it is readily seen that $p^{2k+1} || M_G(pA+B(x-1))$.

For $k=2$ this just leaves the measures with $b=5$.  For $p=3,5$ or $7$ we have
$$ M_{D_{2p^2}}\left( \frac{x^{(p+1)/2}-1}{x-1} + y\left( \frac{x^{(p+1)/2}-1}{x-1} -x^{p-1}\right) \right)=p^5. $$

\end{proof}

\end{document}